\documentclass[12pt,british]{article}
\usepackage{color}
\usepackage{babel}
\usepackage{amsmath,amsthm}
\usepackage{amsfonts}
\usepackage{amssymb}
\usepackage[utf8]{inputenc}
\usepackage{multicol}
\usepackage[all]{xy}
\usepackage{enumerate}
\usepackage{wrapfig}
\usepackage{geometry}
\usepackage{array}
\newtheorem{theorem}{Theorem}[section]

\newtheorem{corollary}[theorem]{Corollary}

\newtheorem{lemma}[theorem]{Lemma}
\newtheorem{proposition}[theorem]{Proposition}

\theoremstyle{definition}

\newtheorem{notation}[theorem]{Notation}

\newtheorem{definition}[theorem]{Definition}

\newtheorem{remark}[theorem]{Remark}

\newcommand{\n}{\mathbb{N}}

\usepackage[%pdftex,                %
bookmarks=true,
breaklinks=true,
bookmarksnumbered = true,%     % Signets numérotés
colorlinks= true,%     % Liens en couleur
urlcolor= green,
anchorcolor = yellow,
citecolor=blue,%  % Couleur des liens externes
]{hyperref}                   % Utilisation de HyperTeX

\begin{document}

\title{Homotopy minimal periods for fiber maps on circle bundles over the circle}
\author{WESLEM LIBERATO SILVA
~\footnote{Departamento de Ciências Exatas, Universidade Estadual de Santa Cruz, Rodovia Jorge Amado, Km 16, Bairro Salobrinho, CEP 45662-900, Ilhéus-BA, Brazil.
e-mail: \texttt{wlsilva@uesc.br}}
\and
RAFAEL MOREIRA DE SOUZA
~\footnote{Universidade Estadual de Mato Grosso do Sul, 
Cidade Universit\'aria de Dourados - Caixa postal 351 - CEP: 79804-970, Dourados-MS, Brazil.
e-mail: \texttt{moreira@uems.br}}
%thanks{}
}
%
%\subjclass{Primary 55M20; Secondary 37C25}

%\keywords{Periodic points, fiber bundle, minimal periods}

%\date{December 26, 2015}

\maketitle

%%%%%%%%%%%%%%%%%%%%%%%%%%%%%%%%%%%%%%%%%%%%%%%%%%%%%%%%%%%%%%%%%%%%%%%%%%%%%%%%%%%%%%%%%%%%%%%%%%%%%%%%%%%%%%%%%%%%%%%%%%%%%%%%%%%%%%%%
% ABSTRACT
%%%%%%%%%%%%%%%%%%%%%%%%%%%%%%%%%%%%%%%%%%%%%%%%%%%%%%%%%%%%%%%%%%%%%%%%%%%%%%%%%%%%%%%%%%%%%%%%%%%%%%%%%%%%%%%%%%%%%%%%%%%%%%%%%%%%%%%%

\begin{abstract}

Given a fiber bundle $Y\rightarrow M\xrightarrow{p}B$ and a fiber map $f: M\rightarrow M$ over $B$, we introduce the notion of fiberwise homotopy minimal periods, denoted by $H_BPer(f).$ This invariant records the periods that occur among representatives of the fiberwise homotopy class of $f.$ We investigate the case in which both the base and the fiber are circles. Up to isomorphism, the corresponding total spaces are the torus and the Klein bottle. Using Nielsen theory for fiber maps and Nielsen-type periodic numbers, we obtain a complete classification of $H_{S^1}Per(f)$ for fiber maps of the torus and the Klein bottle over $S^1$. In particular, we identify the exceptional cases in which some periods can be removed by fiberwise homotopy, including the case $H_{S^1}Per(f)=\mathbb{N}\setminus{2}.$
\end{abstract}

%
%
%
%

%\begingroup%Localising the change to `thefootnote'.
%\renewcommand{\thefootnote}{}%Removing the footnote symbol.
\footnotetext{ Primary 55M20; Secondary 37C25 \\	
Key words: Periodic points, fiber bundle, minimal periods}
%\endgroup
%

%\noindent

%%%%%%%%%%%%%%%%%%%%%%%%%%%%%%%%%%%%%%%%%%%%%%%%%%%%%%%%%%%%%%%%%%%%
%INTRODUCTION
%%%%%%%%%%%%%%%%%%%%%%%%%%%%%%%%%%%%%%%%%%%%%%%%%%%%%%%%%%%%%%%%%%%%

\section{Introduction}

%%%%%%%%%%%%%%%%%%%%%%%%%%%%%%%%%%%%%%%%%%%%%%%%%%%%%%%%%%%%%%%%%%%%

Let $f: X \rightarrow X$ be a self-map of a topological space and let $n$ be a positive integer. A point $x \in X$ is called an $n$-periodic point of $f$ if $f^n(x)=x.$ Its minimal period is the smallest positive integer $n$ for which $f^n(x)=x.$ We denote by $P_n(f)$ the set of points of minimal period $n,$ and by
$$
Per(f)=\{n\in\mathbb N\mid P_n(f)\neq\varnothing\}
$$
the set of minimal periods of $f.$

Periodic points play an important role in the study of dynamical systems. In many situations, topological methods provide information about the existence of periodic orbits that is stable under perturbation. One such invariant is the set of homotopy minimal periods
$$
HPer(f)=\bigcap_{g\sim f}Per(g),
$$
where the symbol $\sim$ means a homotopy. Thus $n\in HPer(f)$ means that every map in the homotopy class of $f$ cannot simultaneously avoid period $n.$

A natural setting in which to refine this problem is provided by fiber-preserving maps. Let $ Y\rightarrow M\xrightarrow{p}B$ 
be a fiber bundle. A map $f: M \rightarrow M $ is called a fiber map over $B$ if $ p\circ f=p. $ In other words, $f$ preserves every fiber of the bundle. In this setting it is natural to restrict the allowed homotopies to homotopies through fiber maps. We therefore write $ f \sim_B g$ if $f$ and $g$ are homotopic over $B.$

The purpose of this paper is to study the periodic information that remains when only fiberwise homotopies are allowed. We make the following definition.

\begin{definition}
	
Let $ Y\rightarrow M\xrightarrow{p}B$ be a fiber bundle and let $ f: M\rightarrow M$ be a fiber map over $B.$ The set of fiberwise homotopy minimal periods of $f$ is
$$
H_BPer(f)=\bigcap_{g\sim_B f}Per(g).
$$	
\end{definition}

Thus, $n\in H_BPer(f)$ if and only if every fiberwise homotopy class representative of $f$ has a periodic point of minimal period $n.$

The distinction between ordinary and fiberwise homotopy is essential. In general, $ HPer(f)\subseteq H_BPer(f),$ because the family of fiberwise homotopies is contained in the family of ordinary homotopies. The inclusion may be strict. For example, when the Klein bottle $K$ is regarded as a circle bundle over $S^1$, the identity map satisfies  $ HPer(Id_K)=\varnothing, $
whereas the fiberwise homotopy class of the identity contains no fixed point free map and hence $ H_{S^1}Per(Id_K)={1}.$
This example shows that restricting the admissible homotopies to those preserving the fibration can change the set of homotopy minimal periods.

The study of fixed points and periodic points of fiber-preserving maps is closely related to Nielsen theory over a base. Nielsen numbers for fiber maps were introduced in \cite{G-K-09} using techniques involving fixed point classes, indices and bordism. More recently, Nielsen-type periodic numbers for maps over a base were introduced in \cite{S-S-25} and related to the detection of fiberwise homotopy minimal periods. These results provide effective tools for determining $H_BPer(f).$

The contribution of the present paper is different from merely computing Nielsen numbers of iterates. Our objective is to determine the set of periods that is forced by a fiberwise homotopy class. In particular, we use Nielsen theoretic information to obtain a complete classification of the sets $H_{S^1}Per(f)$ for all fiber maps of the total spaces of circle bundles over the circle.

We restrict our attention to the case $B=S^1 $ and $ Y=S^1.$ Up to isomorphism of fiber bundles, there are exactly two such bundles: the torus $T^2$ and the Klein bottle $K.$ Consequently, the problem reduces to determining the fiberwise homotopy minimal periods for maps of these two spaces.

For the torus, every fiber map over $S^1$ is fiberwise homotopic to a map
$f_{r,s}(\langle x,y \rangle) = \langle x^r y^s,y\rangle, r,s\in\mathbb{Z}.$ Our first main result gives a complete classification.

\begin{theorem} \label{main-result}

Let $f: T^2\rightarrow T^2$ be a fiber map over $S^1.$ Then $f$ is homotopic over $S^1$ to a map $f_{r,s}$ of the form $f_{r,s}(\langle x,y \rangle) = \langle x^r y^s,y\rangle $ for some $r,s\in\mathbb{Z}.$ Moreover,
$$H_{S^1}Per(f) = \left\{
\begin{array}{ll}
	\varnothing, & (r,s)=(1,0),\\
	\{1\}, & r=-1\text{ or }r=0, \\
	\mathbb{N}\setminus{2}, & r=-2,\\
	\mathbb{N}, & |r|>1,\ r\neq -2.\\
\end{array} \right.
$$

\end{theorem}

In particular, the set $H_{S^1}Per(f)$ is completely determined by the fiber degree $r,$ except for the trivial distinction between the homotopy classes with $r=1.$ 
For the Klein bottle, every fiber map is fiberwise homotopic to a map
$f_{r,s}(\langle z, t \rangle) =  \langle z^r e^{2\pi ist},t\rangle, r\in\mathbb{Z} $ and $s\in \{0,1\}.$ We obtain the following classification.

\begin{theorem} \label{main-result-2}

Let $f: K\rightarrow K$ be a fiber map over $S^1.$ Then $f$ is homotopic over $S^1$ to a map $f_{r,s},$ where $r\in\mathbb{Z}$ and $s\in{0,1}.$ Moreover,
$$
H_{S^1}Per(f) = \left \{
\begin{array}{ll}
	\{1\},& (r,s)\in \{(-1,0),(-1,1),(0,0),(1,0),(0,1)\}, \\
	\{2\},&(r,s)=(1,1), \\
	\mathbb{N} \setminus\{2\},& r=-2, \\
	\mathbb{N}, & |r|>1,\ r \neq -2.\\
\end{array} \right.
$$
\end{theorem}

The exceptional case $r=-2$ is particularly noteworthy: period $2$ can be removed by a fiberwise homotopy, whereas every period $n>2$ is forced. This phenomenon occurs for both the torus and the Klein bottle.

The paper is organized as follows. In Section 2 we recall the necessary facts from Nielsen theory for fiber maps and establish criteria that will be used to detect fiberwise homotopy minimal periods. Section 3 is devoted to the torus and contains the proof of Theorem \ref{main-result}. Section 4 treats the Klein bottle and proves Theorem \ref{main-result-2}.

%%%%%%%%%%%%%%%%%%%%%%%%%%%%%%%%%%%%%%%%%%%%%%%%%%%%%%%%%%%%%%%%%%%%

\section{Preliminaries and generalities}

%%%%%%%%%%%%%%%%%%%%%%%%%%%%%%%%%%%%%%%%%%%%%%%%%%%%%%%%%%%%%%%%%%%%
%%%%%%%%%%%%%%%%%%%%%%%%%%%%%%%%%%%%%%%%%%%%%%%%%%%%%%%%%%%%%%%%%%%%%%

In this section we present some useful results for fiber maps over $B$ in order to compute $H_{B}Per(f).$ Throughout this section $Y \to M \to B$ will be a fiber bundle with base $B,$ fiber $Y$ and $f: M \to M$ is a fiber map  over $B.$ 

Note that If $f\sim_B g,$ then $ H_BPer(f) = H_BPer(g).$ Indeed, the relation $\sim_B$ is an equivalence relation, and therefore $f$ and $g$ have exactly the same fiberwise homotopy class of representatives. Thus $H_BPer$ is an invariant of the fiberwise homotopy class.

\begin{proposition} \label{prop-restriction}
Let $f: M \to M$ be a fiber map over $B$. Let $\widetilde{f}: Y \to Y$ be the restriction of $f$ to the fiber $Y,$ that is, $\widetilde{f} = f|_{Y}.$ Then
$$Per(\widetilde{f}) \subset Per(f).$$
\end{proposition}
\begin{proof}
Suppose that $n \in Per(\widetilde{f}).$ This implies that exists $x \in Y$ such that $\widetilde{f}^{n}(x) = x$ and $\widetilde{f}^{k}(x) \neq x$ for all $k < n.$ For each $m \in \mathbb{N}$ and $y \in Y$ we have 
$$f^{m}(\iota(y)) = \iota(\widetilde{f}^{m}(y)), $$
where $\iota: Y \to M$ is the inclusion. Therefore, we have 
$f^{n}(\iota(x)) =\iota(\widetilde{f}^{n}(x))  = \iota(x)$ and $f^{k}(\iota(x)) \neq \iota(x)$ for all 
$k < n.$ This implies that $n \in Per(f).$
\end{proof}

The following is also an elementary but useful relation between ordinary and fiberwise homotopy minimal periods.

\begin{proposition}
Let $f: M \to M$ be a fiber map over $B.$ We have $HPer(f) \subseteq H_BPer(f).$
\end{proposition}
\begin{proof}
Just observe that $\{g \mid g \sim_{B} f  \}  \subseteq \{g \mid g \sim f  \}.$ Here the symbol $\sim$ means a ordinary homotopy. 
\end{proof}

The inclusion may be strict, as shown by the identity map of the Klein bottle considered as a fiber map over $S^1.$

In \cite{G-K-09} was developed a Nielsen theory to study fixed point (or coincidences) of maps over $B.$ The Nielsen number, $N_B(f),$ and the  Reidmeister number, $R_B(f)$, over $B$ were defined.  

\begin{definition}[Nielsen classes over $B$] \label{nielsen-b}
	Let $f: M \to M$ be a fiber map over $B.$ Two points $x, y \in Fix(f)$ are called Nielsen equivalent over $B$ if there exist a path $\lambda: I \to M$ with $\lambda(0) = x$ and $\lambda(1) = y$ and a homotopy $H: I \times I \to M$ such that $H(t,0) = \lambda(t),$ $H(t,1) = f\circ \lambda(t),$ $H(0,s) = x,$ $H(1,s) = y$ and for each $t$ the image $H(\{t\} \times I)$ lies in the fiber $p^{-1}(p(\lambda(t))).$   
\end{definition}

Using the equivalence relation above we can split $Fix(f)$ into disjoint classes called by Nielsen classes over $B.$ To each Nielsen class one associates an index over $B$, as in \cite[Section 5]{G-K-09}. The Nielsen number of $f$ over $B$ is the number of Nielsen classes with nonzero index. A class with nonzero index is called essential.

\begin{definition}
	Let $f: M \to M$ be a fiber map over $B.$ We denote by 
	$$C_B(f) = min\{\#\pi_0(Fix(g))| g \sim_{B} f \} \,\,\,\,\,\,\, and $$
	$$C_BP_n(f) = min\{\# \pi_0(P_n(g))| g \sim_{B} f \}.$$
\end{definition}

\begin{definition}[Index assumption]
	A fiber map $f: M \to M$ over $B$ is said to satisfies the {\it index assumption } if for each $k,r \in \mathbb{N},$  $C$ a fixed point class of $f^{k}$ and $\overline{C}$ a fixed point class of  $f^{kr}$ which contains $C,$ then  $\overline{C}$ essential implies $C$ essential.
\end{definition}

This assumption guarantees that essential Nielsen classes persist under iteration in the way required for periodic point detection. For more details of Nielsen classes of the iterates of a fiber map $f,$ see \cite{S-S-25}.

\begin{proposition} \label{prop-hpers1}
	If $f: M \to M$ is a fiber map over $B$ satisfying the index assumption and 
	$$\displaystyle \sum_{\frac{n}{k} : prime} N_B(f^{k}) < N_B(f^{n})$$ 
	then $n \in H_B Per(f).$ 
\end{proposition}
\begin{proof}
Suppose, by contradiction, that $ n\notin H_BPer(f).$ Then there exists a fiber map $g \sim_B f$ such that $n\notin Per(g).$ Consequently, every periodic point of $g$ fixed by $g^n$ has minimal period a proper divisor of $n.$

Let $C$ be an essential Nielsen class of $g^n.$ Since $C$ contains a periodic point whose minimal period is a proper divisor of $n,$ there exists a prime divisor $q$ of $n$ such that the corresponding periodic point belongs to a fixed point class of $g^q.$ By the index assumption, the fixed point class of $g^n$ containing this class is essential only if the corresponding class at level $q$ is essential. Thus every essential Nielsen class of $g^n$ is accounted for by an essential Nielsen class of $g^q,$ for some prime divisor $q$ of $n.$ Therefore
$$\displaystyle \sum_{\frac{n}{k} : prime} N_B(g^{k}) \geq N_B(g^{n}).$$
Since Nielsen numbers are invariant under fiberwise homotopy, we obtain
$$\displaystyle \sum_{\frac{n}{k} : prime} N_B(f^{k}) \geq N_B(f^{n})$$
contradicting the hypothesis. Hence $ n\in H_BPer(f).$
\end{proof}

From the above proposition we obtain;

\begin{lemma} \label{lemma-hpers1}
	Let $f: M \to M$ a map over $B$ satisfying the index assumption. If  $N_B(f^{n+1})/N_B(f^{n})$ is well defined and greater than or equal to 2 for all $n \geq 1,$ then $H_B Per(f) = \mathbb{N}.$   
\end{lemma}
\begin{proof}
The hypothesis gives $ N_B(f^n)\geq 2N_B(f^{n-1}) $ for every $n\geq2.$ Hence, for $1\leq k\leq n-1,$
$$
N_B(f^k) \leq \frac{N_B(f^{n-1})}{2^{n-1-k}}.
$$
Every prime divisor $q$ of $n$ satisfies $q\leq n/2.$ Therefore
$$ \displaystyle\sum_{\substack{\frac{n}{q} : prime}} N_B(f^q) \leq \sum_{k=1}^{n-1}N_B(f^k) \leq N_B(f^{n-1}) \sum_{j=1}^{n-1}\frac1{2^j}
< N_B(f^{n-1}).$$
Since $ N_B(f^n)\geq2N_B(f^{n-1}),$ we obtain
$ \displaystyle \sum_{\substack{\frac{n}{q} : prime}} N_B(f^q) <N_B(f^n).$
Proposition \ref{prop-hpers1} therefore implies $ n\in H_BPer(f).$ Since $n$ is arbitrary,
$ H_BPer(f)=\mathbb N.$
\end{proof}

\begin{definition}
	Let $f: M \to M$ be a fiber map over $B.$  We define by induction; 
	$$A_{1}(f)=N_B(f) \ \ and \ \ A_{n}(f) = N_B(f^{n}) - \sum_{k|n, \ k<n} A_{k}(f).$$ 	
\end{definition}

The next result follows from \cite[Theorem 5.4]{S-S-25}.
\begin{theorem} \label{theorem-An}
	Let $f: M \to M$ be a fiber map over $B$ and $n \in \mathbb{N}.$ Let $n = p_1^{\alpha_1}p_2^{\alpha_2}\cdots p_t^{\alpha_t}$ be its prime factorization, where $t \geq 1,$ every $p_i$ prime and every $\alpha_i \geq 1.$ Then
	$$ A_n(f) = \displaystyle \sum_{\substack{\alpha_{j}-1 \leq k_j \leq \alpha_j \\ 1  \leq j \leq t}}
	(-1)^{(\alpha_1+\alpha_2+\cdots+\alpha_t)-(k_1+k_2+\cdots+k_t)} N_B(f^{p_1^{k_1}p_2^{k_2}\cdots p_t^{k_t}}).$$	
\end{theorem}

In \cite{S-S-25} was define a Nielsen type periodic number denoted by $N_BP_n(f)$ satisfying $N_BP_n(f) \leq C_BP_n(f).$ If $N_BP_n(f) > 0$ then $n \in H_{B} Per(f).$  

\begin{theorem} \cite[Theorem 5.3]{S-S-25} \label{th-nbpn}
	If $f: M \to M$ is an $n$-toral fiber map over $B$ such that for every $m | n;$ $0 \neq N_B(f^{m}) = R_B(f^{m})$ then $$A_n(f) = N_BP_n(f).$$ In this case if $A_n(f) \neq 0$ then $n \in H_{B} Per(f).$
\end{theorem}

For the definition of a $n$-toral map over $B,$ see \cite[Definition 5.1]{S-S-25}.

\begin{remark}
A fibration $Y \to M \stackrel{p}{\to} S^{1}$ is also a fiber bundle because the base $S^{1}$ is a compact space. Up to isomorphism of fiber bundles, there are only two fiber bundles with base and the fiber $S^{1},$ $S^{1} \to M \stackrel{p}{\to} S^{1},$ where $M$ is the torus or the Klein bottle. For more details see \cite{G-K-09}. From now on we will focus in these two fiber bundles. 
\end{remark}

The next result is an adaptation of \cite[Theorem 1.3]{G-K-09} for the fixed point case.

\begin{theorem}\cite[Theorem 1.3]{G-K-09} \label{nielsen-number-s1}

\begin{enumerate}[(1)]

\item Suppose $M = \mathbb{T}.$ Given $f: \mathbb{T} \to \mathbb{T}$ a fiber map over $S^{1}$ then $f$ is  homotopic, over $S^{1},$ to the map  $f_{r,s}$ defined by $f_{r,s}(<x,y>) = <x^{r}y^{s},y>$ for some $r,s \in \mathbb{Z}.$  We have;
$$ C_{S^{1}}(f) = N_{S^{1}}(f) = R_{S^{1}}(f) = gcd\{r-1, s\} \,\, if \,\, (r,s) \neq (1,0). $$
$$ 0 = C_{S^{1}}(f) = N_{S^{1}}(f) \neq R_{S^{1}}(f) = \infty  \,\, if \,\, (r,s) = (1,0). $$

\item Suppose $M = \mathbb{K}.$ Given $f: \mathbb{K} \to \mathbb{K}$ a fiber map over $S^{1}$ then $f$ is homotopic, over $S^{1},$ to the map  $f_{r,s}(<z,t>) = <z^{r}[t]^{s},t>$ for some $r \in \mathbb{Z},$ where $[t]$ denotes the image of $t$ under the natural projection $[0,1] \to S^{1} = \dfrac{[0,1]}{0 \sim 1},$ and $s \in \{0,1\}.$ 	
$$\textrm{For} \,\,\, r \neq 1 \,\,\, \textrm{we have;} \,\,\,\,\,\,  C_{S^{1}}(f) = N_{S^{1}}(f) = R_{S^{1}}(f) = \left\{
\begin{array}{lll}
	\frac{|r-1|}{2} & if &\hbox{$r$ is odd and $s = 1$} \\	
	\left[\frac{|r-1|}{2} \right] + 1 &  & else \\	
\end{array}	\right. $$	
$$  \textrm{For} \,\,\, r = 1 \,\,\, \textrm{we have;} \,\,\,\,\,\,\, \infty =  R_{S^{1}}(f) \neq C_{S^{1}}(f) = N_{S^{1}}(f) =   \left\{
\begin{array}{lll}
	0 & if & s = 1  \\	
	1 & if & s = 0 \\	
\end{array}	\right. $$	
	
\end{enumerate}

\end{theorem}

\begin{corollary} \label{corol-nielsen-s1}
If $f_{r,s}: \mathbb{K} \to \mathbb{K}$ is the fiber map over $S^{1}$ defined in Theorem \ref{nielsen-number-s1}, item (2), then $N_{S^{1}}(f_{r,s})$ is given by:
	
\begin{enumerate}[(i)]
		
\item \label{thn-2-1} $\dfrac{|r-1|}{2} + 1,$ if $r$ is odd, $r \neq 1$ and $s=0.$ 
		
\item \label{thn-2-2} $\dfrac{|r-1|}{2},$ if $r$ is odd, $r \neq 1$ and $s=1.$

\item \label{thn-2-3} $\dfrac{|r-1|+1}{2},$ if $r$ is even.

\item \label{thn-2-4} $1,$ if $r=1$ and $s=0.$

\item \label{thn-2-5} $0,$ if $r=1$ and $s=1.$
		
\end{enumerate}	
\end{corollary}
\begin{proof}
Just observe that $$ \left[\frac{|r-1|}{2} \right] + 1 = \left\{
	\begin{array}{lll}
		\frac{|r-1|}{2} + 1 & if & \hbox{$r$ is odd} \\	
		\frac{|r-1|+1}{2}  & if  & \hbox{$r$ is even} \\	
	\end{array}	\right.
	$$
\end{proof}

%%%%%%%%%%%%%%%%%%%%%%%%%%%%%%%%%%%%%%%%%%%%%%%%%%%%%%%%%%%
%%%%%%%%%%%%%%%%%%%%%%%%%%%%%%%%%%%%%%%%%%%%%%%%%

\section{The torus case}

%%%%%%%%%%%%%%%%%%%%%%%%%%%%%%%%%%%

In this section we will proof the Theorem \ref{main-result}, that is, we will give a description for $H_{S^{1}}Per(f)$ for each fiber map $f: \mathbb{T} \to \mathbb{T}$ over $S^{1}.$ In this section we will use the notation $<x,y>  = (x,y)$ for a point in $\mathbb{T}.$

We consider $S^{1}$ as a subset of the complex numbers and $\mathbb{T}$ the trivial fiber bundle $S^{1} \times S^{1}.$ Given a fiber map $f: \mathbb{T} \to \mathbb{T}$ over $S^{1}$ we can deform $f$ to the map $f_{r,s}: \mathbb{T} \to \mathbb{T} $ defined by $ f_{r,s}(x,y) = (x^{r}y^{s},y)$
for some $r,s \in \mathbb{Z},$ see \cite[Section 3]{G-L-V-V}.

\begin{lemma} \cite[Theorem 1.3, item (i)]{G-K-09}
The map $f_{r,s}$ can be deformed over $S^{1}$ to a fixed point free map if and only if $r=1$ and $s=0.$ Therefore, $H_{S^{1}}Per(Id) = \emptyset.$	
\end{lemma}

%\begin{proposition} \label{thn-1}
%Suppose $(r,s) \neq (1,0).$ If $f_{r,s}: T \to T$ is the fiber-preserving %map defined above then;
%
%\begin{enumerate}[(i)]
%
%\item \label{thn-i-1} $N_B(f_{r,s}) = gcd\{r-1, s \}.$
%
%\item \label{thn-i-2} $Fix(f_{r,s})$ is composed by exactly $gcd\{r-1,s\}$ %disjoint circles.
%\end{enumerate}	
%\end{proposition}
%\begin{proof}
%The item \eqref{thn-i-1} follows from \cite[Theorem 1.3, item (i)]{G-K-09}. %For item \eqref{thn-i-2}, we have that from \cite[Theorem 1.3, item %(i)]{G-K-09} that the minimum path components in $Fix(f_{r,s})$ is exactly %$gcd\{r-1,s\}.$ Is not difficult to see that each path component is %composed by a disjoint circle.  
%\end{proof}

From now on we will suppose $(r,s) \neq (1,0).$ In this situation  $Fix(f_{r,s})$ is composed by exactly $gcd\{r-1,s\} = N_{S^{1}}(f_{r,s})$ disjoint circles, see \cite[Proposition 3.4]{G-L-V-V}. Because that we call $f_{r,s}$ by {\it minimal} map.

\begin{notation} \label{notation-1}
For simplification reasons we will use in this work the following notation; $$\sigma(n,r) = \displaystyle{\sum_{i=0}^{n-1}r^{i}}$$ for each $n \in \mathbb{N}$ and $r \in \mathbb{Z}.$  
\end{notation}

\begin{proposition} \label{prop-nielsen-fn}
	
If $f_{r,s}: \mathbb{T} \to \mathbb{T}$ is given by  $f_{r,s}(x,y)=(x^{r}y^{s},y)$ then

\begin{enumerate}[(1)]
	
\item $f^{n}_{r,s}$ is given by $f^{n}_{r,s}(x,y)=\Big(x^{r^{n}}.y^{s.\sigma(n,r)},y\Big)$ for each $n\in \n.$ 

\item $N_{S^{1}}(f^{n}_{r,s}) = |\sigma(r,n)|N_{S^{1}}(f_{r,s}) \neq 0$   if   $(r,s)\neq(1,0)$ and $(r,s,n)\neq(-1,0,2k)$, $k \in \mathbb{N}.$

\item If $(r,s) = (1,0)$ or $(r,s,n) = (-1,0,2k)$ then $N_{S^{1}}(f^{n}_{r,s}) = 0.$

\end{enumerate}

\end{proposition}

\begin{proof}
${\bf (1)}$ Follows from \cite[Lemma 5.6]{S-S-25}. ${\bf (2)}$ We have  $f^{n}_{r,s}(x,y)=f_{r^{n},s.\sigma(n,r)}(x,y)$ for each $n \in \mathbb{N}$ and $r \in \mathbb{Z}.$ If $(r,s)\neq(1,0)$ and $(r,s,n)\neq(-1,0,2k)$, $k \in \mathbb{N},$ we have $$N_{S^{1}}(f^{n}_{r,s}) = \displaystyle{gcd\{r^{n}-1,s.\sigma(n,r)\}=|\sigma(n,r)|gcd\{r-1,s\}=|\sigma(n,r)|N_{S^{1}}(f_{r,s})},$$
because
$$r^{n}-1=(r-1).\Big(\displaystyle{\sum_{i=0}^{n-1}r^{i}}\Big)=(r-1).\sigma(n,r)\neq 0.$$
The conditions above implies $N_{S^{1}}(f^{n}_{r,s}) \neq 0.$ 
The item {\bf (3)} is trivial.
\end{proof}

\begin{corollary} \label{corol-index}
If $(r,s)\neq(1,0)$ and $(r,s,n)\neq(-1,0,2k)$, $k \in \mathbb{N},$ then $f_{r,s}$ has the index assumption.
\end{corollary}
\begin{proof}
Follows from Theorem \ref{nielsen-number-s1} that in this situation all fixed point class are essential, therefore the result follows.
\end{proof}

\begin{proposition} \label{prop-n-m}
If $n$ is a multiple of $m$ and $r^{m} -1 \neq 0$  then $N_{S^{1}}(f^{n}_{r,s})$ is a multiple of $N_{S^{1}}(f^{m}_{r,s})$. 
\end{proposition}
\begin{proof}
This follows from the below equation, and the fact that $N_{S^{1}}(f^{n}_{r,s}) = |\sigma(n,r)|N_{S^{1}}(f_{r,s}).$ 
$$ \sigma(n,r) = \displaystyle{\sum_{i=1}^{n}r^{i-1}} = \left(\displaystyle{\sum_{j=1}^{\frac{n}{m}}r^{m(j-1)}} \right)
\displaystyle{\sum_{i=1}^{m}r^{i-1}} = \left(\displaystyle{\sum_{j=1}^{\frac{n}{m}}r^{m(j-1)}} \right) \sigma(m,r) = \left( \frac{r^{n} -1}{r^{m}-1} \right) \sigma(m,r).$$
\end{proof}

\begin{proposition} \label{prop-hperf}
Let $f: \mathbb{T} \to \mathbb{T}$ be a fiber map over $S^{1}$ such that $f \sim_{S^{1}} f_{r,s}$ where $(r,s)\neq(1,0)$ and $(r,s,n)\neq(-1,0,2k)$, $k\in\n.$ Then $f$ is a $n-$toral map over $S^{1}$ and therefore if  $A_n(f) \neq 0.$ Thus $n \in H_{S^{1}}Per(f).$ 
\end{proposition}
\begin{proof}
Follows from \cite[Proposition 5.7]{S-S-25}.
\end{proof}

\begin{proposition}
Let $f: \mathbb{T} \to \mathbb{T}$ be a fiber map over $S^{1}$ such that $f \sim_{S^{1}} f_{r,s}$ with $|r| > 1$ and $r \neq -2.$ Then $Per(f) = \mathbb{N}.$ 	
\end{proposition}
\begin{proof}
Since $f \sim_{S^{1}} f_{r,s}$ then $deg(f|_{S^{1}}) = r.$ From \cite[Theorem 6.1.4]{J-M-06} we have $Per(f|_{S^{1}}) = \mathbb{N}.$ By Proposition 
\ref{prop-restriction} follows that $Per(f) = \mathbb{N}.$
\end{proof}

%\begin{proposition}
%Let $f: \mathbb{T} \to \mathbb{T}$ be a fiber-preserving map such that $f %\simeq_{S^{1}} f_{r,s}$ where $(r,s)\neq(1,0)$ and $(r,s,n)\neq(-1,0,2k)$, %$k\in\n.$ Then,
%
%\begin{enumerate}[(i)]
%	
%\item If $A_{n}(f) = 0$ then $H_{S^{1}}P_n(f) = \emptyset.$ 
%
%\item $H_{S^{1}}P_n(f)$ is composed by $A_{n}(f)$ disjoint circles. 
%
%\item $N_{S^{1}}(f^{n}) = \sum_{k|n} A_{k}(f).$
%
%\item $H_{S^{1}}Per(f) = \{n \in \mathbb{N}| A_{n}(f) \neq 0 \}.$
%
%\item $H_{S^{1}}Per(f) \subset \{n \in \mathbb{N}| N_{S^{1}}(f^{n}) \neq 0 %\}.$ 
%\end{enumerate}
%\end{proposition}

\begin{proposition} Let $f: \mathbb{T} \to \mathbb{T}$ be a fiber map over $S^{1}$ such that $f \sim_{S^{1}} f_{r,s}$ with $r\geq 2$ or $r\leq -3.$ Then $\dfrac{N_{S^{1}}(f^{n+1})}{N_{S^{1}}(f^{n})} \geq 2$ for all $n\geq1,$ and therefore $H_{S^{1}}Per(f) = \mathbb{N}.$
\end{proposition}
\begin{proof}
Let $r$ be an integer such that $r\geq 2$ or $\leq -3$ and $n\geq1$. Then, $$N_{S^{1}}(f^{n})=\underbrace{|\sigma(n,r)|}_{\neq0}\underbrace{gdc\{r-1,s\}}_{\neq0}.$$
We observe that $r^{n}-1=\left(1+r+\dots+r^{n-1}\right)(r-1)=\sigma(n,r)(r-1)\Rightarrow r^{n}=\sigma(n,r)(r-1)+1$. So,
$$\dfrac{N_{S^{1}}(f^{n+1})}{N_{S^{1}}(f^{n})} = \dfrac{|\sigma(n+1,r)|}{|\sigma(n,r)|}=\left|\dfrac{1+r+\dots+r^{n-1}+r^{n}}{1+r+\dots+r^{n-1}}\right|=\left|1+\dfrac{r^{n}}{\sigma(n,r)}\right|=\left|1+(r-1)+\dfrac{1}{\sigma(n,r)}\right|$$
$$=\left|r+\dfrac{1}{\sigma(n,r)}\right|> 2.$$
By Proposition \ref{prop-hpers1} and Corollary \ref{corol-index} we obtain $H_{S^{1}}Per(f) = \mathbb{N}.$
\end{proof}

\begin{proposition}
Let $f: \mathbb{T} \to \mathbb{T}$ be a fiber map over $S^{1}$ such that $f \sim_{S^{1}} f_{r,s}$ with $r=0.$ Then $H_{S^{1}}Per(f)= \{1\}.$ 
\end{proposition}
\begin{proof}
In fact, by Proposition \ref{prop-nielsen-fn} we have $f^{n}_{0,s} = f_{0,s}$ for all $n.$ Since $A_1(f) = N_{S^{1}}(f) \neq 0$ and $A_n(f_{0,s}) = 0$ for all $n \geq 2,$ then the result follows by Proposition \ref{prop-hperf}.
\end{proof}

\begin{proposition} 
Let $f: \mathbb{T} \to \mathbb{T}$ be a fiber map over $S^{1}$ such that $f \sim_{S^{1}} f_{1,s}$ with $s \neq 0.$ Then $H_{S^{1}}Per(f_{1,s}) = \mathbb{N}.$
\end{proposition}
\begin{proof}
 Let $n=p_{1}^{\alpha_{1}}\dots p_{l}^{\alpha_{l}}$ is its prime factorization. From \cite[Proposition 5.8]{S-S-25} we have $A_n(f) = |s|p_{1}^{\alpha_{1}-1}\dots p_{l}^{\alpha_{l}-1}(p_{1}-1)\dots(p_{l}-1) > 0.$ By Proposition \ref{prop-hperf} the result follows.
\end{proof}

\begin{proposition} 
	Let $f: \mathbb{T} \to \mathbb{T}$ be a fiber map over $S^{1}$ such that $f \sim_{S^{1}} f_{-1,s}.$ Then $H_{S^{1}}Per(f) = \{1\}.$
\end{proposition}
\begin{proof}
By Proposition \ref{prop-nielsen-fn} we have
$$ f^{n}_{-1,s} = \left\{
\begin{array}{lll}
	f_{1,0} =  Id & if &\hbox{$n$ is even} \\	
	f_{-1,s} & if & \hbox{$n$ is odd} \\	
\end{array}	\right. $$	
Thus we obtain $H_{S^{1}}Per(f) = H_{S^{1}}Per(f_{-1,s}) = \{1\}.$ 	
\end{proof}

\begin{remark}
The computation of $H_{S^{1}}Per(f_{-2,s})$ follows from the following propositions, which were proved for all $r < -1,$  and Proposition \ref{prop-hperf}. 
The computation was split into three cases; $r < -1$ and $n $ odd, $r< -1$ and $n=2l$  but $4$ does not divide $n,$ and the last $r< -1$ and $4$ divides $n.$
\end{remark}

\begin{proposition}\label{proposition-An>0,2}
Let $f: \mathbb{T} \to \mathbb{T}$ be a fiber map over $S^{1}$ such that $f \sim_{S^{1}} f_{r,s}$ with  $r < - 1$ and $n$  odd. Then $A_{n}(f)>0$.	
\end{proposition}
\begin{proof}
	Let $n=p_{1}^{\alpha_{1}}.p_{2}^{\alpha_{2}}\dots p_{l}^{\alpha_{l}}$ and $m=p_{1}^{\alpha_{1}-1}.p_{2}^{\alpha_{2}-1}\dots p_{l}^{\alpha_{l}-1}$ be its prime factorization, such that $3\leq p_{1}<p_{2}<\dots<p_{l}$. By Proposition \ref{prop-n-m}, we have 
	$$ N_{S^{1}}(f^{n})=N_{S^{1}}(f^{m}).\left(\dfrac{1-r^{n}}{1-r^{m}}\right) \ and \  N_{S^{1}}(f^{\frac{n}{p_{j}}})=N_{S^{1}}(f^{m}).\left(\dfrac{1-r^{\frac{n}{p_{j}}}}{1-r^{m}}\right), \ for \ all \ j=1,2,\dots l.$$ 
	Then,
	$$\begin{array}{ccl}
		N_{S^{1}}(f^{n}) - \displaystyle \sum_{1  \leq j \leq l} N_{S^{1}}(f^{\frac{n}{p_{j}}}) & = & N_{S^{1}}(f^{m}).\left(\dfrac{1-r^{n}}{1-r^{m}}\right)- \displaystyle\sum_{1  \leq j \leq l}N_{S^{1}}(f^{m}).\left(\dfrac{1-r^{\frac{n}{p_{j}}}}{1-r^{m}}\right)\\
		& \geq & N_{S^{1}}(f^{m}).\left(\dfrac{1-r^{n}}{1-r^{m}}\right)-l.N_{S^{1}}(f^{m}).\left(\dfrac{1-r^{\frac{n}{p_{1}}}}{1-r^{m}}\right)\\
		& = & \dfrac{N_{S^{1}}(f^{m})}{1-r^{m}}.\left((1-r^{n})-l.(1-r^{\frac{n}{p_{1}}})\right)\\
		& = & \dfrac{N_{S^{1}}(f^{m})}{1-r^{m}}.\left(-r^{\frac{n}{p_{1}}}.(r^{m.(p_{1}-1).p_{2}\dots p_{l}}-l)-l+1\right)>0.
	\end{array}$$
	By the same idea, we have
	$$\displaystyle N_{S^{1}}(f^{\frac{n}{p_{1} . p_{2}}}) - \sum_{ 3  \leq i \leq l} N_{S^{1}}(f^{\frac{n}{p_1 . p_2 . p_i}}) > 0 ,$$
	and, by induction, we have
	$$\displaystyle \sum_{1  \leq i < j \leq l} N_{S^{1}}(f^{\frac{n}{p_{i}.p_{j}}}) - \displaystyle \sum_{1  \leq i < j < k \leq l} N_{S^{1}}(f^{\frac{n}{p_{i}.p_{j}.p_{k}}}) > 0.$$
	By the same idea, for $2t+1\leq n$, we have
	$$\displaystyle \sum_{\substack{\alpha_{j}-1 \leq k_j \leq \alpha_j \\ 1  \leq j \leq l\\ (\alpha_1+\alpha_2+\cdots+\alpha_l)-(k_1+k_2+\cdots+k_l)= 2t}} N_{S^{1}}(f^{p_1^{k_1}p_2^{k_2}\cdots p_l^{k_l}}) - \displaystyle \sum_{\substack{\alpha_{j}-1 \leq k_j \leq \alpha_j \\ 1  \leq j \leq l\\ (\alpha_1+\alpha_2+\cdots+\alpha_l)-(k_1+k_2+\cdots+k_l)=2t+1}} N_{S^{1}}(f^{p_1^{k_1}p_2^{k_2}\cdots p_l^{k_l}}) > 0.$$
	By Theorem \ref{theorem-An} we obtain $A_{n}(f)>0.$
\end{proof}

\begin{proposition}\label{proposition-An>0,3}
Let $f: \mathbb{T} \to \mathbb{T}$ be a fiber map over $S^{1}$ such that $f \sim_{S^{1}} f_{r,s}$ with  $r < - 1.$ Let $n=2.p_{2}^{\alpha_{2}}\dots p_{l}^{\alpha_{l}}$ is its prime factorization, such that $3\leq p_{2}<\dots<p_{l}.$ Then $A_{n}(f)>0$ if and only if $(n,r)\neq (2,-2).$	
\end{proposition}
\begin{proof}
	Firstly, we suppose $(n,r)=(2,-2).$ Then $N_{S^{1}}(f^{2}_{-2,s})=N_{S^{1}}(f_{-2,s})$ and therefore  $A_{2}(f)=N_{S^{1}}(f^{2}_{-2,s})-N_{S^{1}}(f_{-2,s})=0.$
	
	On the other hand, if $(n,r)\neq (2,-2)$ let $p_{1}=2$ and $m=p_{2}^{\alpha_{2}-1}\dots p_{l}^{\alpha_{l}-1}$ be its prime factorization, such that $3\leq p_{2}<\dots<p_{l}$. By Proposition \ref{prop-n-m}, we have 
	$$ N_{S^{1}}(f^{n})=N_{S^{1}}(f^{m}).\left(\dfrac{r^{n}-1}{1-r^{m}}\right) \ and \  N_{S^{1}}(f^{\frac{n}{2}})=N_{S^{1}}(f^{m}).\left(\dfrac{1-r^{\frac{n}{2}}}{1-r^{m}}\right).$$ 
	Then,
	$$\begin{array}{ccl}
		N_{S^{1}}(f^{n}) - \displaystyle \sum_{1  \leq j \leq l} N_{S^{1}}(f^{\frac{n}{p_{j}}}) & \geq & N_{S^{1}}(f^{m}).\left(\dfrac{r^{n}-1}{1-r^{m}}\right)-l.N_{S^{1}}(f^{m}).\left(\dfrac{1-r^{\frac{n}{2}}}{1-r^{m}}\right)\\
		& = & \dfrac{N_{S^{1}}(f^{m})}{1-r^{m}}.\left((r^{n}-1)-l.(1-r^{\frac{n}{2}})\right)\\
		& = & \dfrac{N_{S^{1}}(f^{m})}{1-r^{m}}.\left(r^{\frac{n}{2}}.(r^{m.p_{2}\dots p_{l}}+l)-l-1\right)>0.
	\end{array}$$
	By the same idea, we have
	$$\displaystyle N_{S^{1}}(f^{\frac{n}{p_{1} . p_{2}}}) - \sum_{ 3  \leq i \leq l} N_{S^{1}}(f^{\frac{n}{p_1 . p_2 . p_i}}) > 0 ,$$
	and, by induction, we have
	$$\displaystyle \sum_{1  \leq i < j \leq l} N_{S^{1}}(f^{\frac{n}{p_{i}.p_{j}}}) - \displaystyle \sum_{1  \leq i < j < k \leq l} N_{S^{1}}(f^{\frac{n}{p_{i}.p_{j}.p_{k}}}) > 0.$$
	By the same idea, for $2t+1\leq n$, we have
	$$\displaystyle \sum_{\substack{\alpha_{j}-1 \leq k_j \leq \alpha_j \\ 1  \leq j \leq l\\ (\alpha_1+\alpha_2+\cdots+\alpha_l)-(k_1+k_2+\cdots+k_l)= 2t}} N_{S^{1}}(f^{p_1^{k_1}p_2^{k_2}\cdots p_l^{k_l}}) - \displaystyle \sum_{\substack{\alpha_{j}-1 \leq k_j \leq \alpha_j \\ 1  \leq j \leq l\\ (\alpha_1+\alpha_2+\cdots+\alpha_l)-(k_1+k_2+\cdots+k_l)=2t+1}} N_{S^{1}}(f^{p_1^{k_1}p_2^{k_2}\cdots p_l^{k_l}}) > 0.$$
	From Theorem \ref{theorem-An} we have $A_{n}(f)>0.$
\end{proof}

\begin{proposition}
Let $f: \mathbb{T} \to \mathbb{T}$ be a fiber map over $S^{1}$ such that $f \sim_{S^{1}} f_{r,s}$ with $r < -1$ and $4|n.$ Then $A_{n}(f)>0.$	
\end{proposition}
\begin{proof} Following the proof of Proposition \ref{proposition-An>0,2} we have $p_{1}=2$, $\alpha_{1}\geq 2$, $2|m$ and $r^{m.(2-1).p_{2}\dots p_{s}}>0$. Thus, the proof follows analogous to the proof of Proposition \ref{proposition-An>0,2}.
\end{proof}

The proof of Theorem \ref{main-result} follows from the above results.

%The Theorem \ref{main-result} follows by above results. From the above %computations we also have:
%\begin{theorem}
%Let $f: T \to T$ be a fiber-preserving map such that $f \simeq_{S^{1}} %f_{r,s}$ where $(r,s)\neq(1,0)$ and $(r,s,n)\neq(-1,0,2k)$, $k\in\n.$ Then 
%$m \notin H_{S^{1}}Per(f)$ if and only if $C(f^{m}) = C(f^{\frac{m}{p}})$ %for some prime factor $p$ of $m.$
%\end{theorem}

%Let $d_{0}=0<d_{1}=1<d_{2}<\dots<d_{l}<d_{l+1}=n$ be the list of divisors %of $n$. Then, 
%$$\begin{array}{ccl}
%A_{n} & = & C(f^{n}_{r,s}) - \displaystyle{\sum_{i=1}^{l} %C(f^{d_{i}}_{r,s})} = C(f_{r,s}).\Big( \displaystyle{\sum_{i=0}^{n-1} %r^{i}} - \displaystyle{\sum_{j=1}^{l} \sum_{j=0}^{d_{j-1}}r^{k}}\Big)\\
%& & \\
%& = & C(f_{r,s}).\Big( (1-l)+(2-l).\displaystyle{\sum_{i=d_{1}}^{d_{2}-1} %r^{i}}+\dots+(-1)\displaystyle{\sum_{i=d_{l-2}}^{d_{l-1}-1} %r^{i}}+(0)\displaystyle{\sum_{i=d_{l-1}}^{d_{l}-1} %r^{i}}+\displaystyle{\sum_{i=d_{l}}^{n-1} r^{i}}\Big)\\
%& = & %C(f_{r,s}).\displaystyle{\sum_{j=1}^{l+1}(j-l).\sum_{i=d_{j-1}}^{d_{j}-1} %r^{i}}.
%\end{array}$$

%%%%%%%%%%%%%%%%%%%%%%%%%%%%%%%%%%%%%%%%%%%%%%%%%%%%%%%%%%%%%%%%%%%%%

%%%%%%%%%%%%%%%%%%%%%%%%%%%%%%%%%%%%%%%%%%%%%%%%%%%%%%

\section{The Klein bottle case}

In this section we will proof the Theorem \ref{main-result-2}.

Let $\mathbb{K} = \dfrac{S^{1} \times I}{(z,0) \sim (\overline{z},1)}$ be the Klein bottle, where $S^{1}$ is considered as a subset of the complex numbers. 
If $f: \mathbb{K} \to \mathbb{K}$ is a fiber map over $S^{1}$ then $p_2 \circ f = p_2,$ where $p_2 : \mathbb{K} \to S^{1}$ is given by $p_2(<z,t>) = [t].$
\begin{equation} 
\xymatrix{  \mathbb{K} \ar[rr]^-{f} \ar[dr]_-{p_2}    &     &  \mathbb{K} \ar[dl]^-{p_2}  \\
& S^{1}  &  }
\end{equation}
Here we are considering $S^{1} = \dfrac{[0,1]}{0 \sim 1}.$ We have that $f$ is given by $f(<z,t>) =<F(z,t), t>$ for some map $F: \mathbb{K} \to S^{1}.$ 
For each pair of integers $(r,s)$ let $f_{r,s}: \mathbb{K} \to \mathbb{K}$ be the fiber map over $S^{1}$ defined by $$f_{r,s}(<z,t>) = <z^{r}[t]^{s},t>$$
where $[t]$ denotes the image of $t$ under the natural projection $[0,1] \to S^{1} = \dfrac{[0,1]}{0 \sim 1}.$ 

\begin{proposition} \cite[Proposition 3.6]{G-L-V-V}
If $f: \mathbb{K} \to \mathbb{K}$ is a fiber map over $S^{1}$ then $f$ is homotopic over $S^{1}$ to the map $f_{r,s}$ for some integer $r$ and some $s \in \{0,1\}.$
\end{proposition}

Let us consider $\pi_1(\mathbb{K}) = \langle \alpha, \beta \mid \alpha \beta \alpha \beta^{-1} = 1 \rangle.$ Since $p_2  = p_2 \circ f$ then $f_{\#}: \pi_1(\mathbb{K}) \to \pi_1(\mathbb{K})$ is given by 
$$f_{\#}(\alpha) = \alpha^{r} \,\,\,\,\, and \,\,\,\,\, f_{\#}(\beta) = \alpha^{s} \beta$$
for some $r, s \in \mathbb{Z}.$

\begin{proposition} \label{prop-n-k}
Given $f: \mathbb{K} \to \mathbb{K}$ a fiber map over $S^{1}$ and $n \in \mathbb{N}$ we have;
\begin{enumerate}[(1)]

\item $f^{n}_{\#}(\alpha) = \alpha^{r^{n}} \,\,\,\,\, and \,\,\,\,\, f^{n}_{\#}(\beta) = \alpha^{s \sigma(n,r)} \beta.$

\item For $r^{n} \neq 1$  we have 
$$C_{S^{1}}(f^{n}) = N_{S^{1}}(f^{n}) = R_{S^{1}}(f^{n}) = \left\{
\begin{array}{lll}
	\frac{|r^{n}-1|}{2} & if &\hbox{$r$ is odd and $s \sigma(n,r) = 1$} \\	
	\left[\frac{|r^{n}-1|}{2} \right] + 1 &  & else \\	
\end{array}	\right. $$	
For $r^{n} = 1$  we have 
$$ \infty =  R_{S^{1}}(f^{n}) \neq C_{S^{1}}(f^{n}) = N_{S^{1}}(f^{n}) =   \left\{
\begin{array}{lll}
	0 & if & s \sigma(n,r) = 1  \\	
	1 & if & s \sigma(n,r) = 0 \\	
\end{array}	\right. $$

\item If $g_{r,s}: \mathbb{K} \to \mathbb{K}$ is given by $g_{r,s}(<z,t>) = <z^{r} exp(2\pi i st), t>,$ for $s \in \{0,1\}$ then $g_{r,s} \sim_{S^{1}} f$ and 
$$ N_{S^{1}}(g_{r,s}) = \# \{\pi_0(Fix(g_{r,s})) \}$$

%\item $$ N_{S^{1}}(g^{n}_{r,s}) = \# \{\pi_0(Fix(g^{n}_{r,s})) \}$$	
%for some $r  ?????$

%\item $H_{S^{1}}P_n(f)= \#\pi_0(P_n(g_{r,s})).$
	
\end{enumerate}
\end{proposition}
\begin{proof}
{\bf (1)} Follows by a simple induction. {\bf (2)} Follows by Theorem \ref{nielsen-number-s1}. {\bf (3)} Follows from \cite[section 3]{G-L-V-V}. 
\end{proof}

%\begin{proposition} \label{prop-hperf-k}
%Let $f: \mathbb{K} \to \mathbb{K}$ be a fiber map over $S^{1}$ such that $f %\sim_{S^{1}} f_{r,s}$ where $r^{n} \neq 1,$ then  $n \in H_{S^{1}}Per(f)$ %if %and only if $A_n(f) \neq 0.$ 
%\end{proposition}
%\begin{proof}
%From the conditions of statement of the proposition and Proposition %\ref{prop-n-k} we have that all fixed point classes are essential. The map %$g_{r,s}$ is minimal for each $n$, that is, $\#\pi_0(Fix(g^{n}_{r,s})) =  %N_{S^{1}}(f^{n}_{r,s})$ by Proposition \ref{prop-n-k}, item $(3).$ We note %that   
%$$H_{S^{1}}P_n(f)= \#\pi_0(P_n(g_{r,s}))  = \#\pi_0 \left( %Fix(g^{n}_{r,s})  - \left\{ \bigcup_{\substack{ k | \\
%			k \,\, \textrm{divides}\,\, n}} Fix(g^{k}_{r,s}) \right \} %\right) = A_n(g_{r,s}) = A_n(f).$$ 
%Therefore the result follows.
%\end{proof}

In the rest of this section will focus to compute $H_{S^{1}}Per(f_{r,s}).$   

\begin{proposition} Let $f: \mathbb{K} \to \mathbb{K}$ be a fiber map over $S^{1}$ such that $f \sim_{S^{1}} f_{r,s}.$ If $r \geq 2$ or $r\leq -4$ then $\dfrac{N_{S^{1}}(f^{n+1})}{N_{S^{1}}(f^{n})}\geq 2$ for all $n\geq1,$ and therefore $H_{S^{1}}Per(f) = \mathbb{N}.$
\end{proposition}
\begin{proof}
	From Proposition \ref{prop-n-k} (2), for $r \geq 2$ or $r\leq -4$ and for all $n\in\mathbb{N}$, we have $N_{S^{1}}(f^{n}_{r,0})=N_{S^{1}}(f^{n}_{r,1})$ because $r^{n}\neq 1$ and $s\sigma(n,r)\neq1$. Hence, we do not need different calculations for $s=0$ or $s=1$. 	
	
	The computation of this proposition was split into six cases: Case 1 for $r$ odd and $r\geq3$; Case 2 for $r$ odd, $r\leq -3$ and $n$ even; Case 3 for $r$ odd $r\leq -5$ and $n$ odd; Case 4 for $r$ even and $r\geq2$; Case 5 for $r$ even, $r\leq-2$ and $n$ even; Case 6 for $r$ even, $r\leq-4$ and $n$ odd.
	
	Case 1: $\dfrac{N_{S^{1}}(f^{n+1})}{N_{S^{1}}(f^{n})}=\dfrac{|r^{n+1}-1|+2}{|r^{n}-1|+2}=\dfrac{r^{n+1}+1}{r^{n}+1}$. If $n=1$, so $\dfrac{r^{2}+1}{r+1}= (r-1) + \dfrac{2}{r+1}>r-1\geq2$. If $n>1$, so $\dfrac{r^{n+1}+1}{r^{n}+1}=r+\dfrac{1-r}{r^{n}+1}>r-1\geq2.$
	
	Case 2: $\dfrac{N_{S^{1}}(f^{n+1})}{N_{S^{1}}(f^{n})}=\dfrac{|r^{n+1}-1|+2}{|r^{n}-1|+2}=\dfrac{-r^{n+1}+3}{r^{n}+1}=\underbrace{-r}_{=|r|}+\dfrac{r+3}{r^{n}+1}>|r|-1\geq2.$
	
	Case 3: $\dfrac{N_{S^{1}}(f^{n+1})}{N_{S^{1}}(f^{n})}=\dfrac{|r^{n+1}-1|+2}{|r^{n}-1|+2}=\dfrac{r^{n+1}+1}{-r^{n}+3}$. If $n=1$, so $\dfrac{r^{2}+1}{-r+3}= (-r-3) + \dfrac{10}{-r+3}=|r|-3+\dfrac{10}{|r|+3}$. Thus, for $r=-3$ we have $\dfrac{r^{2}+1}{-r+3}=\dfrac{10}{6}<2$. But, for $r\leq -5$ we have $\dfrac{r^{2}+1}{-r+3}=|r|-3+\dfrac{10}{|r|+3}>|r|-3\geq2$. Moreover, if $n\geq3$, so $\dfrac{r^{n+1}+1}{-r^{n}+3}=-r+\dfrac{3r+1}{-r^{n}+3}>|r|-1\geq2.$
	
	Case 4: $\dfrac{N_{S^{1}}(f^{n+1})}{N_{S^{1}}(f^{n})}=\dfrac{|r^{n+1}-1|+1}{|r^{n}-1|+1}=\dfrac{r^{n+1}}{r^{n}}=r\geq2.$
	
	Case 5: $\dfrac{N_{S^{1}}(f^{n+1})}{N_{S^{1}}(f^{n})}=\dfrac{|r^{n+1}-1|+1}{|r^{n}-1|+1}=\dfrac{-r^{n+1}+2}{r^{n}}=-r+\dfrac{2}{r^{n}}>|r|\geq2.$
	
	Case 6: $\dfrac{N_{S^{1}}(f^{n+1})}{N_{S^{1}}(f^{n})}=\dfrac{|r^{n+1}-1|+1}{|r^{n}-1|+1}=\dfrac{r^{n+1}}{-r^{n}+2}$. If $n=1$, so $\dfrac{r^{2}}{-r+2}= (-r-2) + \dfrac{4}{-r+2}=|r|-2+\dfrac{4}{|r|+2}$. Thus, for $r\leq -4$ we have $\dfrac{r^{2}}{-r+2}=|r|-2+\dfrac{4}{|r|+2}>|r|-2\geq2$. Moreover, if $n\geq3$ and $r\leq -4$, so $\dfrac{r^{n+1}}{-r^{n}+2}=-r+\dfrac{2r}{-r^{n}+2}>|r|-1\geq2.$	
	
Under the conditions of statement of the proposition $f$ satisfies the index assumption because in this situation all classes are essential. Now the result follows from Lemma \ref{lemma-hpers1}.
\end{proof}

%\begin{proposition}
%Given $f_{r,0}: \mathbb{K} \to \mathbb{K}$ a fiber map over $S^{1}$ with  $r\in \z-\{1\}$ odd and $n\in\n.$ We have $A_{n}(f)=\dfrac{|r^{n}-r^{n-1}|}{2},$ and therefore $n \in  H_{S^{1}}Per(f_{r,0}).$
%\end{proposition}
%\begin{proof}
%We suppose $r > 0.$ We have $A_{1}(f_{r,0}) = N_{S^{1}}(f_{r,0}) = \dfrac{|r -1|}{2}+1.$ Note that $A_{2}(f)= N_{S^{1}}(f^{2}_{r,0})- N_{S^{1}}(f_{r,0})=\dfrac{r^{2}-1}{2}+1-\left(\dfrac{r-1}{2}+1\right)=\dfrac{r^{2}-r}{2}.$ By induction we suppose that $A_{n}(f)=\dfrac{r^{n}-r^{n-1}}{2}.$ Then 
%$$\begin{array}{ccl}
%A_{n+1}(f) & = & N_{S^{1}}(f^{n+1}_{r,0})-\displaystyle{\sum_{i=1}^{n}}A_{i}(f)=\dfrac{r^{n+1}-1}{2}+1-\left(\dfrac{r-1}{2}+1+\left(\sum_{i=2}^{n}\dfrac{r^{i}-r^{i-1}}{2}\right)\right)\\
%& = & \dfrac{r^{n+1}-1}{2}+1-\left(\dfrac{r^{n}-1}{2}+1\right)=\dfrac{r^{n+1}-r^{n}}{2}.
%\end{array}$$
%The prove is analogous for $r < 0.$ Now the result follows from Proposition \ref{prop-hperf-k}.
%\end{proof}

\begin{proposition}
$H_{S^{1}}Per(f_{1,0}) = \{1\}.$ 
\end{proposition}
\begin{proof}
By \cite{G-87} the identity map $Id = f_{1,0}: \mathbb{K} \to \mathbb{K}$ cannot be deformed over $S^{1}$ to a fixed point free map. Since $Id^{n} = Id$ then the result follows.	
\end{proof}

\begin{proposition}
$H_{S^{1}}Per(f_{1,1}) = \{2\}.$ 
\end{proposition}
\begin{proof}
By \cite[Proposition 3.8]{G-L-V-V} the map $f_{1,1}$ is homotopic over $S^{1}$ to the map $h(<z,t>) = $ $<z  exp(\pi i), t>.$ The map $h$ has no fixed point, and this implies $1 \notin H_{S^{1}}Per(f_{1,1}).$ However 
$$ h^{n}(<z,t>) = \left\{  \begin{array}{lll}
	h^{2}(<z,t>) & if & n \,\, is \,\, even \\
	h(<z,t>) & if & n \,\, is \,\, odd \\
\end{array} \right. $$
Therefore $H_{S^{1}}Per(f_{1,1}) = H_{S^{1}}Per(h)  = \{2\}.$
\end{proof}

\begin{proposition} $H_{S^{1}}Per(f_{0,s}) = \{1\}$. 
\end{proposition}

\begin{proof}
From Corollary \ref{corol-nielsen-s1} \ref{thn-2-3} $N_{S^{1}}(f_{0,s})=1$, so $1\in H_{S^{1}}Per(f_{0,s})$. From Proposition \ref{prop-n-k} (3), $f_{0,s}\sim_{S^{1}} g_{0,s}$, such that $g_{0,s}(<z,t>) = < z^{0}exp(2\pi i s t), t>$. 
$$\begin{array}{cl}
\text{For} \ s=0 & <z,t>=g_{0,0}(<z,t>) = <1,t> \Leftrightarrow z=1 \ \text{and} \\
 & <z,t>=g^{2}_{0,0}(<z,t>) = <1,t> \Leftrightarrow z=1.\\
\text{For} \ s=1 & <z,t>=g_{0,1}(<z,t>) = <exp(2\pi i t),t> \Leftrightarrow z=exp(2\pi i t) \ \text{and} \\
 & <z,t>=g^{2}_{0,1}(<z,t>) = <exp(2\pi i t),t> \Leftrightarrow z=exp(2\pi i t).
\end{array}$$
Thus, $Fix(g_{0,s}^{2})=Fix(g_{0,s})$. Consequently, $Fix(g_{0,s}^{n})=Fix(g_{0,s})$, for $n\in\mathbb{N}$. Therefore, if $n>1$ then $n\notin Per(g_{0,s})$ and $n\notin H_{S^{1}}Per(f_{0,s})$. 
\end{proof}

\begin{proposition} $H_{S^{1}}Per(f_{-1,0}) = \{1\}$. 
\end{proposition}

\begin{proof}

From Corollary \ref{corol-nielsen-s1} \ref{thn-2-1} $N_{S^{1}}(f_{-1,0})=2$, so $1\in H_{S^{1}}Per(f_{-1,0})$. To prove that $2\notin H_{S^{1}}Per(f_{-1,0})$ the Klein bottle is obtained by the quotient $\displaystyle \dfrac{I \times I}{\substack{(x,0) \sim (1-x,1) \\(0,y) \sim (1,y)}}$ and we consider 
$$f\left(<x,y>\right) = <\dfrac{1}{8}\sin(2\pi x)+\dfrac{1}{2},y>.$$
We observe that $f_{\#}(\alpha)=\alpha^{-1}$, $f_{\#}(\beta)=\beta$ and $Fix(f)=\{<\frac{1}{2},y>; \ y\in I\}.$ Furthermore, 
$$ f^{2}(<x,y>)= f\left(<\dfrac{1}{8}\sin(2\pi x)+\dfrac{1}{2},y>\right) = <\dfrac{1}{8}\sin\left(2\pi \left(\dfrac{1}{8}\sin(2\pi x)+\dfrac{1}{2}\right)\right)+\dfrac{1}{2},y> $$
$$=<\dfrac{1}{8}\sin\left(\dfrac{\pi}{4}\sin(2\pi x)+\pi\right)+\dfrac{1}{2},y> =<-\dfrac{1}{8}\sin\left(\dfrac{\pi}{4}\sin(2\pi x)\right)+\dfrac{1}{2},y>$$
and $Fix(f^{2})= Fix(f)$. Therefore, $2\notin Per(f)$ and $2\notin H_{S^{1}}Per(f_{-1,0}).$

\

Returning to $f_{-1,0}(<z,t>) = <z^{-1},t>$,  we have $$f^{3}_{-1,0}(<z,t>)=f_{-1,0}(<z,t>) \hspace{0.3cm} \text{and} \hspace{0.3cm} f^{4}_{-1,0}(<z,t>)=f^{2}_{-1,0}(<z,t>).$$ So, $n\notin Per(f_{-1,0})$ if $n\geq3$. Hence, $H_{S^{1}}Per(f_{-1,0})=\{1\}.$
\end{proof}

\begin{proposition} $H_{S^{1}}Per(f_{-1,1}) = \{1\}$. 
\end{proposition}

\begin{proof}

From Corollary \ref{corol-nielsen-s1} \ref{thn-2-2} $N_{S^{1}}(f_{-1,1})=1$, so $1\in H_{S^{1}}Per(f_{-1,1})$. From Proposition \ref{prop-n-k} (3), $f_{-1,1}\sim_{S^{1}} g_{-1,1}$, such that $g_{-1,1}(<z,t>) = < z^{-1}exp(2\pi i t), t>$. Observe that
$$ g^{2}_{-1,1}(<z,t>) = g_{-1,1}(<z^{-1}exp(2\pi i t),t>)=<(z^{-1}exp(2\pi i t))^{-1}exp(2\pi i t),t>=<z,t>.$$
Therefore,
$$ g^{n}_{-1,1}(<z,t>) = \left\{  \begin{array}{lll}
	<z,t> & if & n \,\, is \,\, \text{even} \\
	g_{-1,1}(<z,t>) & if & n \,\, is \,\, \text{odd}. \\
\end{array} \right. $$  
Hence, if $n>2$ then $n\notin Per(g_{-1,1})$ and $n\notin H_{S^{1}}Per(f_{-1,1})$.

To prove that $2\notin H_{S^{1}}Per(f_{-1,1})$ the Klein bottle is obtained by the quotient $\displaystyle \dfrac{I \times I}{\substack{(x,0) \sim (1-x,1) \\(0,y) \sim (1,y)}}$ and we consider 
$$f\left(<x,y>\right)=  <\dfrac{1}{8}\sin(2\pi x)+\frac{1}{2}+y,y> .$$ 
We observe that $f_{\#}(\alpha)=\alpha^{-1}$, $f_{\#}(\beta)=\alpha\beta$ and 
$$Fix(f)=\left\{<x,x+\frac{1}{2}-\frac{1}{8}sin(2\pi x) > ; \ 0\leq x \leq \dfrac{1}{2}\right\} \cup \left\{<x,x-\frac{1}{2}-\frac{1}{8}sin(2\pi x) > ; \ \dfrac{1}{2}\leq x \leq 1\right\}.$$

Moreover, $$f^{2}\left(<x,y>\right)= f\left(<\dfrac{1}{8}\sin(2\pi x)+\dfrac{1}{2}+y,y>\right)=<\dfrac{1}{8}\sin\left(2\pi \left(\dfrac{1}{8}\sin(2\pi x)+\dfrac{1}{2}+y\right)\right)+\dfrac{1}{2}+y,y> $$
$$=<\dfrac{1}{8}\sin\left(\dfrac{\pi}{4}\sin(2\pi x)+\pi+2\pi y\right)+\dfrac{1}{2}+y,y> =<-\dfrac{1}{8}\sin\left(\dfrac{\pi}{4}\sin(2\pi x)+2\pi y\right)+\dfrac{1}{2}+y,y>$$
and $Fix(f^{2})= Fix(f)$. Therefore, $2\notin Per(f)$ and $2\notin H_{S^{1}}Per(f_{-1,1}).$
\end{proof}

In the following results, we shall establish an estimate that provides an upper bound for $\displaystyle \sum_{k=1}^{n-1} N_{S^{1}}(f^{k})$ satisfying $\left(  \displaystyle \sum_{k=1}^{n-1} N_{S^{1}}(f^{k}) \geq \sum_{\frac{n}{k} : prime} N_{S^{1}}(f^{k}) \right)$ while this upper bound remains strictly less than $N_{S^{1}}(f^{n})$. From Proposition \ref{prop-hpers1} follows that $n\in H_{S^{1}} Per(f).$ Some special cases will be treated separately.

\begin{proposition} $H_{S^{1}}Per(f_{-2,s}) = \mathbb{N}-\{2\}$. 
\end{proposition}

\begin{proof}

From Proposition \ref{prop-n-k} (2), for all $n\in\mathbb{N}$, we have $N_{S^{1}}(f^{n}_{-2,0})=N_{S^{1}}(f^{n}_{-2,1})$, because $(-2)^{n}\neq 1$ and $s\sigma(n,-2)\neq1$. Hence, we do not need different calculations for $s=0$ or $s=1$.

We proceed by induction on $n$. For $n=1$ we have $N_{S^{1}}(f_{-2,s})=2$ from Corollary \ref{corol-nielsen-s1}. Hence $1\in H_{S^{1}}Per(f_{-2,s})$.   For $n=2$ and $s=0$, we observe that 
$$\begin{array}{c}
<z,t>=f_{-2,0}(<z,t>) = <z^{-2},t> \Leftrightarrow z^{3}=1 \ \text{and} \\
<z,t>=f^{2}_{-2,0}(<z,t>) = <z^{4},t> \Leftrightarrow z^{3}=1.
\end{array}$$

Thus, $2\notin Per(f_{-2,0})$, and consequently $2\notin H_{S^{1}}Per(f_{-2,0})$. 

For $n=2$ and $s=1$, $f_{-2,1}\sim_{S^{1}} g_{-2,1}$, such that $g_{-2,1}(<z,t>) = <z^{-2} exp(2\pi i t), t>$ from Proposition \ref{prop-n-k} (3).
$$\begin{array}{c}
<z,t>=g_{-2,1}(<z,t>) = <z^{-2}exp(2\pi i t),t> \Leftrightarrow z^{3}=exp(2\pi i t) \ \text{and} \\
<z,t>=g^{2}_{-2,1}(<z,t>) = <z^{4}exp(-2\pi i t),t> \Leftrightarrow z^{-3}=exp(-2\pi i t) \Leftrightarrow z^{3}=exp(2\pi i t).
\end{array}$$

Hence, $2\notin Per(g_{-2,1})$, and consequently $2\notin H_{S^{1}}Per(f_{-2,1})$.

For $n>2$ we shall use Proposition \ref{prop-hpers1} to prove that $n\in H_{S^{1}}Per(f_{-2,s})$. We begin with two preliminary observations. First, for any $m\in\mathbb{N}$, we have 
$N_{S^{1}}(f^{2m}) = \dfrac{|(-2)^{2m}-1| +1}{2}=2^{2m-1}$ and $N_{S^{1}}(f^{2m+1})  =  \dfrac{|(-2)^{2m+1}-1| +1}{2}=2^{2m}+1.$
The following we will show that for all $n>2$ we have;
$$ \displaystyle \sum_{\frac{n}{k} : prime} N_{S^{1}}(f^{k}) \leq \sum_{k=1}^{n-2} N_{S^{1}}(f^{k}) < 2^{n-1} \leq N_{S^{1}}(f^{n}) $$
which implies $n \in  H_{S^{1}}Per(f_{-2,s}).$
We start illustrating the case $n=3.$
$$\displaystyle \sum_{\frac{3}{k} : prime} N_{S^{1}}(f^{k}) = \sum_{k=1}^{3-2} N_{S^{1}}(f^{k})=N_{S^{1}}(f)=2 < 2^{3-1}=4 < 5 = N_{S^{1}}(f^{3}).$$
Similarly, for $n=4$, 
$$\displaystyle \sum_{\frac{4}{k} : prime} N_{S^{1}}(f^{k}) = \sum_{k=1}^{4-2} N_{S^{1}}(f^{k})=2+2 < 2^{4-1}= N_{S^{1}}(f^{4}).$$ 
Assume now that $\displaystyle\sum_{k=1}^{n-2} N_{S^{1}}(f^{k}) < 2^{n-1}.$  If $n=2k$, $k\in\mathbb{N}$,  then

$$\displaystyle\sum_{k=1}^{n-1} N_{S^{1}}(f^{k}) = N_{S^{1}}(f^{n-1})+ \sum_{k=1}^{n-2} N_{S^{1}}(f^{k}) < N_{S^{1}}(f^{2k-1})+2^{2k-1}=2^{2k-2}+1+2^{2k}$$

$$=2^{2k-2}(1+4)+1<2^{2k-2}8=2^{2k+1}<2^{2k+1}+1=N_{S^{1}}(f^{n+1}).$$
 Similarly, if $n=2k+1$, $k\in\mathbb{N}$,  then

$$\displaystyle\sum_{k=1}^{n-1} N_{S^{1}}(f^{k}) = N_{S^{1}}(f^{n-1})+ \sum_{k=1}^{n-2} N_{S^{1}}(f^{k}) < N_{S^{1}}(f^{2k})+2^{2k}=2^{2k-1}+2^{2k}$$

$$=2^{2k-1}(1+2)<2^{2k+1}=N_{S^{1}}(f^{n+1}).$$
\end{proof}

\begin{proposition} $H_{S^{1}}Per(f_{-3,s}) = \mathbb{N}$. 
\end{proposition}

\begin{proof}
From Proposition \ref{prop-n-k} (2), for all $n\in\mathbb{N}$, we have $N_{S^{1}}(f^{n}_{-3,0})=N_{S^{1}}(f^{n}_{-3,1})$, because $(-3)^{n}\neq 1$ and $s\sigma(n,-3)\neq1$. Hence, we do not need different calculations for $s=0$ or $s=1$.

We proceed by induction on $n$. For $n=1$ we have $N_{S^{1}}(f_{-3,s})=3$ from Corollary \ref{corol-nielsen-s1}. Hence $1\in H_{S^{1}}Per(f_{-3,s})$.   For $n=2$, we obtain
$$\displaystyle \sum_{\frac{2}{k} : prime} N_{S^{1}}(f^{k}) = \sum_{k=1}^{2-1} N_{S^{1}}(f^{k})=N_{S^{1}}(f)=3 < 5 = N_{S^{1}}(f^{2}).$$
Consequently, from Proposition \ref{prop-hpers1}, $2\in H_{S^{1}}Per(f_{-3,s}).$ The following we will show by induction that 
$$ \displaystyle \sum_{\frac{m}{k} : prime} N_{S^{1}}(f^{k}) \leq \sum_{k=1}^{m-1} N_{S^{1}}(f^{k})<\dfrac{3^{m}}{2}<N_{S^{1}}(f^{m})  $$
for all $m \in \mathbb{N}.$ 
Assume that $\displaystyle\sum_{k=1}^{n-1} N_{S^{1}}(f^{k}) < \dfrac{3^{n}}{2}.$  By Proposition \ref{prop-n-k} we have 
$$N_{S^{1}}(f^{n+1})= \dfrac{|(-3)^{n+1}-1|}{2}+1 \geq \dfrac{3^{n+1}+1}{2}+1>\dfrac{3^{n+1}}{2}.$$ Note that
$$\displaystyle\sum_{k=1}^{n} N_{S^{1}}(f^{k}) = N_{S^{1}}(f^{n})+ \sum_{k=1}^{n-1} N_{S^{1}}(f^{k}) < N_{S^{1}}(f^{n})+\dfrac{3^{n}}{2}= \dfrac{|(-3)^{n}-1|}{2}+1+\dfrac{3^{n}}{2}$$

$$\leq \dfrac{3^{n}+1}{2}+1+\dfrac{3^{n}}{2}=3^{n} +\dfrac{3}{2}\leq \dfrac{3^{n+1}}{2},$$
 Therefore
$$\displaystyle \sum_{\frac{n+1}{k} : prime} N_{S^{1}}(f^{k}) \leq \sum_{k=1}^{n} N_{S^{1}}(f^{k})<\dfrac{3^{n+1}}{2}<N_{S^{1}}(f^{n+1}).$$
From Proposition \ref{prop-hpers1} we obtain $H_{S^{1}}Per(f_{-3,s}) = \mathbb{N}.$
\end{proof}

The Theorem \ref{main-result-2} follows from the above results.

%%%%%%%%%%%%%%%%%%%%%%%%%%%%%%%%%%%%%%%%%%%%%%%%%%%%%%

%%%%%%%%%%%%%%%%%%%%%%%%%%%%%%%%%%%%%%%%%%%%%%%%%%%%%%
% BIBLIOGRAPHY
%%%%%%%%%%%%%%%%%%%%%%%%%%%%%%%%%%%%%%%%%%%%%%%%%%%%%%

\end{document}